\documentclass[12pt]{article}

\usepackage[margin=1in]{geometry}

\usepackage{amsmath, amssymb, amsthm, amssymb}
\usepackage{graphpap}
\usepackage{amsfonts}
\usepackage{epsfig}
\usepackage{latexsym}
\usepackage{graphicx}
\usepackage{amsbsy}
\usepackage{euscript}           
\usepackage[frame,ps,matrix,arrow,curve,rotate]{xy}
\usepackage{verbatim}
\usepackage{graphics}
\usepackage{enumerate}
\usepackage{color}
\usepackage{tikz}
\usepackage{hyperref}
\usepackage{color}

\newtheorem{theorem}{Theorem}[section]
\newtheorem{lemma}[theorem]{Lemma}

\newtheorem{proposition}[theorem]{Proposition}
\newcounter{claims}[theorem]
\newtheorem{claim}[claims]{Claim}

\theoremstyle{definition}
\newtheorem{definition}[theorem]{Definition}

\theoremstyle{remark}
\newtheorem{remark}[theorem]{Remark}

\newcommand{\mc}[1]{\mathcal{#1}}

\newcommand{\N}{\mathbb{N}}
\newcommand{\Z}{\mathbb{Z}}

\makeatletter
\newcommand\mathcircled[1]{%
  \mathpalette\@mathcircled{#1}%
}
\newcommand\@mathcircled[2]{%
  \tikz[baseline=(math.base)] \node[draw,circle,inner sep=1pt] (math) {$\m@th#1#2$};%
}
\makeatother

\begin{document}

\title{Characterizations of Cancellable Groups}
\author{Matthew Harrison-Trainor\thanks{Supported by an NSERC Banting Fellowship.},  Meng-Che ``Turbo" Ho}

\maketitle

\begin{abstract}
	An abelian group $A$ is said to be cancellable if whenever $A \oplus G$ is isomorphic to $A \oplus H$, $G$ is isomorphic to $H$. We show that the index set of cancellable rank 1 torsion-free abelian groups is $\Pi^0_4$ $m$-complete, showing that the classification by Fuchs and Loonstra cannot be simplified. For arbitrary non-finitely generated groups, we show that the cancellation property is $\Pi^1_1$ $m$-hard; we know of no upper bound, but we conjecture that it is $\Pi^1_2$ $m$-complete.
\end{abstract}

\section{Introduction}

In his book \textit{Infinite Abelian groups} \cite{Kaplansky54}, Kaplansky proposed three test problems as criteria for a ``satisfactory classification theorem''. The third test problem asks: For abelian groups, if $A \oplus G \cong A \oplus H$ and $A$ is finitely generated, is $G \cong H$? In other words, can one cancel finitely generated groups from direct sums? Walker \cite{Walker56} and Cohn \cite{Cohn56} independently answered this question in the affirmative in 1956. In particular, $\mathbb{Z}$ is cancellable, but $\mathbb{Q}$ is not.

All of the groups in this paper will be abelian. In general, we say that $A$ has the \textit{cancellation property}, or is \textit{cancellable}, if it can be cancelled from direct sums. Rotman and Yen \cite{RotmanYen61}, Crawley \cite{Crawley65}, Hs\"u \cite{Hsu62}, and Kaplansky \cite{Kaplansky52} showed that certain other countable but non-finitely-generated groups also have the cancellation property. Finally, Fuchs and Loonstra \cite{FuchsLoonstra71} gave a complete characterization of the cancellable rank-one abelian groups (i.e., the subgroups of $\mathbb{Q}$) in terms of the endormorphism ring of the group.

\begin{definition}\label{def:sr1}
A ring $R$ \emph{has $1$ in the stable range} if whenever $f_1,f_2,g_1,g_2 \in R$ satisfy $f_1 g_1 + f_2 g_2 = 1$, there is $h \in R$ such that $f_1 + f_2 h$ is a unit of $R$.
\end{definition}

\noindent In the rest of this article, we will write $E(G)$ for the endomorphism ring of an abelian group $G$.

\begin{theorem}[Fuchs and Loonstra \cite{FuchsLoonstra71}, see Theorem 8.12 of \cite{Arnold82}]\label{thm:rank-one-char}
Let $G$ be a rank 1 torsion-free abelian group. Then $A$ is cancellable if and only if $G \cong \mathbb{Z}$ or $E(G)$ has 1 is in the stable range.
\end{theorem}

However, this is not a simple condition. The endomorphism ring of any subgroup $G$ of $\mathbb{Q}$ is a localization $\mathbb{Z}_S$ of $\mathbb{Z}$ at a set of primes $S$, and Estes and Ohm \cite{EstesOhm67} highlight two possible extreme cases: (i) $S$ contains all but finitely many primes of $\mathbb{Z}$, and (ii) $S$ contains only finitely many primes. In case (i), they prove that $\mathbb{Z}_S$ has 1 in the stable range, and in case (ii) that 1 is not in the stable range. But then Estes and Ohm give examples which are of neither type (i) nor type (ii) which show that for such rings both possibilities can occur. They conclude that ``the problem of a complete classification of overrings of $\mathbb{Z}$ having 1 in the stable range remains open''. Similarly, Arnold \cite{Arnold82} says that ``rings with 1 in the stable range are not easily characterized''. In Section \ref{sec:sr1}, we will give formal justification to these statements, though the statements we prove are somewhat technical; but the conclusion is that one cannot give a classification better than the statement in Definition \ref{def:sr1}.\footnote{The computability theorist will note that (i) and (ii) allow us to use the fact that $(\Sigma^0_2,\Pi^0_2) \leq_1 (\text{Fin},\text{Cof})$ to show that the property of having stable rank 1 is $\Pi^0_2$-complete, where $\text{Fin}$ and $\text{Cof}$ are the indices of finite and infinite c.e.\ sets respectively. This does not really give the whole picture though. For strongly computable subrings $R$ of $\mathbb{Q}$ (i.e., ones where given $q \in \mathbb{Q}$ we can decide whether $q \in R$) this argument would only give that having 1 in the stable range is $\Pi^0_1$-hard, because the indices of finite and infinite computable sets are both $\Sigma^0_2$. But we show in Section \ref{sec:sr1} having 1 in the stable range is $\Pi^0_2$-hard.} Moreover, we prove formally that there is no better characterization of the cancellable rank 1 torsion-free abelian groups than Theorem \ref{thm:rank-one-char}: to decide whether $G$ has the cancellation property, we must naively check whether $E(G)$ satisfies Definition \ref{def:sr1}.

To formalize these statements, we use the tools of computable structure theory. Suppose $P$ is a certain property of algebraic structures; in our case, $P(A)$ is having the cancellation property, but in general $P(A)$ could be some other property such as being a free group, or being finitely generated. Using a universal Turing machine, produce an effective listing of all computable structures, indexed by natural numbers. The complexity of the property is reflected in the complexity of its \emph{index set}
\[ I_{P} = \{ i \in \mathbb{N} \mid \text{the $i$th computable structure has property $P$}\}. \]
The complexity of $I_P$ can be formally measured using various computability-theoretic hierarchies such as the arithmetical and analytical hierarchies. These hierarchies consist of complexity classes which are the analogues in computability theory to classes such as $P$ and $NP$ in complexity theory. The classes correspond to the number and type of quantifiers required to solve the problem. We show some of the complexity classes below. 

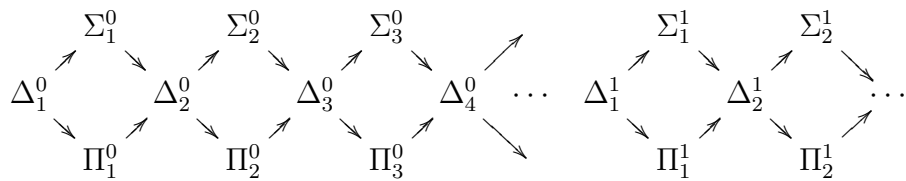
\[
\xymatrix@=7pt{
& \Sigma^0_1 \ar[dr]& &\Sigma_2^0\ar[dr]&& \Sigma^0_3 \ar[dr] &&&& \Sigma^1_1 \ar[dr]& &\Sigma_2^1\ar[dr]& \\
\Delta^0_1 \ar[ur]\ar[dr] &  & \Delta^0_2 \ar[ur]\ar[dr]&&\Delta^0_{3} \ar[ur] \ar[dr] & & \Delta^0_4 \ar[ur] \ar[dr] & \cdots & \Delta^1_1 \ar[ur]\ar[dr] &  & \Delta^1_2 \ar[ur]\ar[dr]&& \cdots \\
& \Pi^0_1 \ar[ur] &&\Pi_2^0\ar[ur]&& \Pi^0_3\ar[ur] &&&& \Pi^1_1 \ar[ur] &&\Pi_2^1\ar[ur]&
}
\]

If $A$ is a finite-rank group, $A$ is cancellable if for all groups $G$ and $H$ and isomorphisms $A \oplus G \to A \oplus H$, there is an isomorphism $G \to H$. This is, on the face of it, a property that requires quantifiers over sets; it is of the form ``for all ... there exist ... such that ...'' where the quantifiers are over groups and functions (i.e., over sets) rather than natural numbers. However, Theorem \ref{thm:rank-one-char} gives a characterization that is simpler. If we unpack the definition of the endomorphism ring (for a finite-rank group), this characterization takes the form ``for all ... there exist ... such that for all ... there exist ... such that ...'' where these quantifiers are all over elements of $G$ (which we can identify with natural numbers). Such a property is $\Pi^0_4$. To see that the characterization can be written in this way, note that the endomorphism ring of $G \subseteq \mathbb{Q}$ can be identified with those elements of $G$ all of whose powers are also in $G$ (see Lemma \ref{lem:end-group}). We show that there is no simpler characterization, by showing that the index set cannot be simpler than $\Pi^0_4$; a simpler characterization would result in the index set being simpler than $\Pi^0_4$.

\begin{theorem}\label{thm:main-fr}
	The index set
	\[ I_{C} = \{ i \mid \text{the $i$th computable group has the cancellation property} \} \]
	of cancellable torsion-free rank 1 groups is $\Pi^0_4$ $m$-complete. Moreover, this relativizes.
\end{theorem}

\noindent Most algebraic properties have characterizations that are either relatively simple ($\Sigma^0_3$ or $\Pi^0_3$ or simpler) or have no characterization ($\Sigma^1_1$ or $\Pi^1_1$ $m$-hard). Thus the cancellation property for rank 1 groups occupies an interesting intermediate space of properties that admit a characterization, but the characterization is complicated.

Formally, we will show that if $S \subseteq \mathbb{N}$ is a $\Pi^0_4$ set, then there is a computable reduction $f$ from $S$ to $I_{C}$; for some computable function $f$ we have:
\[ \text{for all $n$:}\qquad n \in S \Longleftrightarrow f(n) \in I_{C}.\]
If such a reduction exists, then any decision procedure for $I_{C}$ would immediately yield one for $S$; so $I_{C}$ is at least as hard as $S$. Thus characterizing the cancellable torsion-free rank 1 groups is necessarily as complicated as Theorem \ref{thm:rank-one-char}.

In the second half of this paper, we move on to countable infinite-rank groups. Here, again there is no known characterization, and in fact, we know of no upper bound on the complexity of the index set other than that it must be $\Pi^\alpha_2$ for some $\alpha$. Saying that a group $A$ has the cancellation property involves quantifying over all groups $G$ and $H$ of any size, and we know of no argument that says that we can restrict our attention to countable groups. This is quite an unusual and interesting situation.

We do obtain a lower bound: the cancellation property is $\Pi^1_1$ $m$-hard.

\begin{theorem}\label{thm:main-inf}
	The index set of the class of cancellable torsion-free countable groups is $\Pi^1_1$ $m$-hard. Moreover, this relativizes.
\end{theorem}

\noindent This means that there is no simpler definition than one which uses a universal quantifier over countable sets. The proof uses methods from Riggs's \cite{Riggs15} theorem that indecomposability is $\Pi^1_1$ $m$-complete.

We conjecture that: (a) if a group cancels with every countable group, then it cancels with every groups, and so the index set is $\Pi^1_2$; and (b) the class is $\Pi^1_2$ $m$-hard, i.e., that (a) is the best characterization of the cancellation property for infinite-rank groups. 

For groups of finite rank greater than 1, it follows from Theorem \ref{thm:main-fr} that the cancellable groups are $\Pi^0_4$ $m$-hard. We again know of no upper bound.

\section{Background on Number Theory}

We say that $p$ is a quadratic residue modulo $q$ if $p \equiv x^2 \pmod q$ for some $x$. Exactly half of the non-zero elements of $\mathbb{Z} / q \mathbb{Z}$ are quadratic residues modulo $q$, and $1 \equiv 1^2 \pmod q$ is always a quadratic residue. We will make use of the law of quadratic reciprocity, which relates $p$ being a quadratic residue modulo $q$ to $q$ being a quadratic residue modulo $p$.

\begin{theorem}[Quadratic reciprocity]
	Given $p$ and $q$ distinct odd primes:
	\begin{itemize}
		\item If $p \equiv 1 \pmod 4$ or $q \equiv 1 \pmod 4$, then $p$ is a quadratic residue modulo $q$ if and only if $q$ is a quadratic residue modulo $p$.
		\item If $p \equiv q \equiv 3 \pmod 4$, then $p$ is a quadratic residue modulo $q$ if and only if $q$ is not a quadratic residue modulo $p$.
	\end{itemize}
\end{theorem}

\noindent Usually we will apply this in the case where we have a fixed $p$, and we want to choose $q$ such that $p$ is a quadratic residue modulo $q$. By choosing $q \equiv 1 \pmod p$ we ensure that $q$ is a quadratic residue modulo $p$. By also choosing $q \equiv 1 \pmod 4$, by quadratic reciprocity we get that $p$ is a   residue modulo $q$. We can find primes satisfying these congruences by Dirichlet's theorem.

\begin{theorem}[Dirichlet's theorem]
	Given $a$ and $d$ coprime, there are infinitely many primes $p \equiv a \pmod{d}$.
\end{theorem}

\noindent We will use quadratic residues for the following fact: a non-residue cannot equal to a product of residues. Indeed, if $a \equiv b_1^2 b_2^2 \cdots b_n^2 \pmod q$, then $a$ itself is a quadratic residue.

\section{\texorpdfstring{Subgroups of $\mathbb{Q}$}{Subgroups of Q}}

To begin, we identify the endomorphism ring of a rank 1 torsion-free abelian group with a localization of $\mathbb{Z}$.

A \emph{multiplicative set} of a ring is a subset of the ring that is closed under multiplication. A multiplicative set $S\subseteq R$ is said to be \emph{saturated} if for any $a,b \in R$ with $ab \in S$, we have $a,b \in S$. All the multiplicative sets we consider are saturated. A saturated multiplicative set in $\mathbb{Z}$ is determined by the primes it contains.

\begin{lemma}\label{lem:end-group}
Let $G$ be a rank 1 torsion-free abelian group. Then $E(G) \cong \mathbb{Z}_M = \mathbb{Z}\left[ \frac{1}{m} \mid m\in M \right]$ where $M$ is the multiplicative set generated by the primes $p$ which infinitely divide each element of $G$.
\end{lemma}
\begin{proof}
Each endomorphism $f$ of $G$ is determined by $f(1)$, and so it is easy to see that $E(G)$ can be identified with a sub-$\mathbb{Z}$-algebra of $\mathbb{Q}$ which is also a subgroup of $G$. Thus
\[ E(G) \subseteq \mathbb{Z}\left[\frac{1}{p^\infty} \;\Bigl\vert\; \text{$p^n \mid 1$ in $G$ for all $n$}\right].\]
On the other hand, if $p^\infty \mid 1$ in $G$, then the multiplication-by-$\frac{1}{p}$ map is an endomorphism of $\mathbb{Q}$ which restricts to an endomorphism of $G$.
\end{proof}

We will make use of the following lemma which simplifies the conditions for a localization of $\mathbb{Z}$ to have 1 in the stable range.

\begin{lemma}[Estes and Ohm, Lemma 7.3 of \cite{EstesOhm67}]\label{lem:localization}
Let $M$ be a saturated multiplicative set of $\mathbb{Z}$. Then the following are equivalent:
\begin{itemize}
	\item 1 is in the stable range of $\mathbb{Z}_M$.
	\item If $\alpha_1$ and $\alpha_2$ are coprime to each other and relatively prime to each element of $M$, then there is $m \in M$ and $b \in \mathbb{Z}$ such that $\alpha_1 m + \alpha_2 b \in M$.
\end{itemize}
\end{lemma}

We are now ready to prove Theorem \ref{thm:main-fr}.

{
\renewcommand{\thetheorem}{\ref{thm:main-fr}}
\begin{theorem}
The index set of the class cancellable torsion-free rank 1 groups is $\Pi^0_4$ $m$-complete.
\end{theorem}
\addtocounter{theorem}{-1}
}

It will not be hard to see that the proof relativizes.

\begin{proof}
Let $S$ be a $\Pi^0_4$ set. We must construct, for each $x$, a torsion-free rank 1 abelian group $G(x)$ such that
$G(x)$ is cancellable if and only if $x \in S$.

Fix an enumeration $(r_i)_{i > 0}$ of the primes, such that each prime shows up infinitely often. We will make use of sequences of primes with special properties from the following lemma.

\begin{lemma}
There are computable sequences of natural numbers $(a_i)_{i \geq 0}$, primes $(q_i)_{i > 0}$, and disjoint finite sets of primes $(P_{i,j})_{i,j \geq 0}$ such that:
\begin{enumerate}
	\item For each $i$ and $j$, $P_{i,j}$ generates the multiplicative group $(\mathbb{Z}/a_i\mathbb{Z})^{\times}$.
	\item Every $p\in P_{i,j}$ is a quadratic residue in $\mathbb{Z} / q_k \mathbb{Z}$ for $k>i$.
	\item $a_0 = 3$ and $a_{i+1} = a_i q_{i+1} r_{i+1}$. Thus, for every $m\in\mathbb{Z}$, $m \mid a_i$ for some $i$.
	\item $q_i$ is coprime to $a_j$ for $j<i$.
	\item $q_k \notin P_{i,j}$ for any $k,i,j$.
\end{enumerate}
\end{lemma}
\begin{proof}
We define $a_i$, $q_i$, and $P_{i,j}$ recursively. Begin with $a_0 = 3$. At every stage $s$, we define $q_s$ (and hence $a_s$) as well as $P_{i,j}$ for $i+j = s$. When we define $q_s$, we will check that (2) holds for $q_s$ and each $p \in P_{i,j}$ already defined, i.e., those $i + j < s$. When we define $P_{i,j}$ at stage $s = i+j$, we will check that $(2)$ holds between each $p \in P_{i,j}$ and each $q_{s'}$, $i < s' \leq s$. In this way, we verify that (2) holds in all instances. We will verify (5) in the same way.

To find $q_s$, we use Dirichlet's theorem to choose a prime satisfying the equation
\[ q_s \equiv 1 \pmod{{4 \prod_{\substack{p \in P_{i,j}\\i + j < s}}{p} \prod_{i < s}{a_i}}}.\]
Define $a_s = a_{s-1} q_s r_s$. Conditions (3), (4), and (5) are immediate. Note that $q_s$ is a quadratic residue modulo each $p \in P_{i,j}$, $i + j < s$. Since $q_s \equiv 1 \pmod 4$, by quadratic reciprocity, each $p \in P_{i,j}$, $i + j < s$, is a quadratic residue modulo $q_s$. So (2) is maintained.

For each $i + j = s$, let $g_1,\ldots,g_\ell$ be a finite set of generators of $(\mathbb{Z} / a_i \mathbb{Z})^\times$. Then $P_{i,j}$ will consist of new primes $p_1,\ldots,p_\ell$, where $p_k$ is a solution to the following system of congruence equations:
\begin{align*}
p_k &\equiv 1 \pmod{q_{i+1}\cdots q_s} \\
p_k &\equiv g_{k} \pmod{a_i}.\end{align*}
Since $q_{i+1},\ldots,q_s$ are coprime to $a_i$, by the Chinese remainder theorem, this is equivalent to a single congruence equation. By Dirichlet's theorem, there is a prime solution which is larger than $q_1,\ldots,q_i$. (1) and (5) are clear. Since $p_k \equiv 1 \pmod{q_j}$, for $k < j \leq s$, and $1$ is a quadratic residue, (2) is maintained.
\end{proof}

We are now ready to construct $G(x)$. Let $R$ be a computable relation such that
\[ x\in S \Longleftrightarrow \forall y \exists z \forall u \exists v R(x;y,z,u,v).\]
We define $G(x)$ to be the following c.e.\ subgroup of $\mathbb{Q}$:
\[ G(x) = \mathbb{Z}\left[\frac1{p^{m}} \ \Bigl\vert\ p\in P_{i,j}, (\forall i'\le i) (\exists j'\le j)(\forall u\le m)(\exists v) R(x;i',j',u,v) \right] \]
We now have to argue that $G(x)$ is cancellable if and only if $x \in S$.

\begin{remark}\label{rem:z}
	We may assume that for each $x$, $\forall u \exists v R(x;0,0,u,v)$. Thus $G(x)$ is never isomorphic to $\mathbb{Z}$.
\end{remark}

Recall, from Lemma \ref{lem:end-group}, that $E(G(x)) \cong \mathbb{Z}_M$ where $M$ is the multiplicative set generated by the primes which infinitely divide each element of $G$. To begin, we can characterize the set $M$.

\begin{claim}
A prime $p$ is in $M$ if and only if $p \in P_{i,j}$ for some $i$ and $j$ and \[(\forall i'\le i) (\exists j'\le j)(\forall u)(\exists v) R(x;i',j',u,v).\]
\end{claim}
\begin{proof}
Fix $p \in P_{i,j}$. Then $p \in M$ if and only if for every $m$, $\frac{1}{p^m} \in G$. If 
\[(\forall i'\le i) (\exists j'\le j)(\forall u)(\exists v) R(x;i',j',u,v) \]
then it is easy to see from the definition of $G(x)$ that $\frac{1}{p^m} \in G(x)$ for every $m$. On the other hand, if $\frac{1}{p^m} \in G$ for every $m$, then for every $m$,
\[ (\forall i'\le i) (\exists j'\le j)(\forall u\le m)(\exists v) R(x;i',j',u,v). \]
Fix $i' \leq i$. For each $m$, there is $j' \leq j$ such that $(\forall u\le m)(\exists v) R(x;i',j',u,v)$. There are only finitely many such $j'$, so there is some $j' \leq j$ such that, for arbitrarily large $m$, it is true that $(\forall u\le m)(\exists v) R(x;i',j',u,v)$. Thus, for this $j'$, we have $(\forall u)(\exists v) R(x;i',j',u,v)$. So we have shown that
\[(\forall i'\le i) (\exists j'\le j)(\forall u)(\exists v) R(x;i',j',u,v) \]
as desired.
\end{proof}

\begin{claim}\label{cl:2}
If $x \in S$, then $G(x)$ is cancellable.
\end{claim}
\begin{proof}
If $x\in S$, consider an equation $\alpha_1 u \equiv 1 \pmod{\alpha_2}$, with $\alpha_1$ and $\alpha_2$ coprime to each other and to each element of $M$, for which we want to find a solution $u \in M$. Find $a_i$ such that $\alpha_2 \mid a_i$. As $x\in S$, for each $i' \le i$, there is a $j_{i'}$ such that $\forall u \exists v R(x;i',j_{i'},u,v)$. Let $j$ be the largest among the $j_{i'}$. Then $(\forall i' \le i) (\exists j'\le j)(\forall u) (\exists v) R(x;i',j',u,v)$. So for every $p\in P_{i,j}$, $p \in M$. But by (1), the primes in $P_{i,j}$ generate the multiplicative group modulo $a_i$, and so there is $u$ a product of primes in $P_{i,j}$ (and hence $u \in M$) such that $u \equiv \alpha_1^{-1} \pmod{\alpha_2}$. Then $\alpha_1 u \equiv 1 \pmod{\alpha_2}$.

By Lemma \ref{lem:localization}, $E(G(x))$ has 1 in the stable range, and by Theorem \ref{thm:rank-one-char} $G(x)$ is cancellable.
\end{proof}

\begin{claim}\label{cl:3}
If $x \notin S$, then $G(x)$ is not cancellable.
\end{claim}
\begin{proof}
If $x\notin S$, then $\exists i' \forall j' \exists u \forall v \neg R(x;i',j',u,v)$. Fix $i^*$ such that $\forall j' \exists u \forall v \neg R(x;i^*,j',u,v)$. Then for every $i \geq i^*$ and every $j$, $(\forall i'\le i) (\exists j'\le j)(\forall u)(\exists v) R(x;i',j',u,v)$ is not true. So none of the $p \in P_{i,j}$ for $i \geq i^*$ is in $M$. So a prime $p$ can only be in $M$ if $p \in P_{i,j}$ for $i < i^*$. Then by (2) $p$ is a quadratic residue modulo $q_{i^*}$. So in fact every element of $M$ is a quadratic residue modulo $q_{i^*}$. Note also that by (5), each element of $M$ is coprime to $q_{i^*}$.

Pick a prime $\alpha_1 \in \mathbb{Z}$ which is not a quadratic residue in $\mathbb{Z} / q_{i^*} \mathbb{Z}$; thus $\alpha_1$ is coprime to each element of $M$. Let $\alpha_2 = q_{i^*}$. Then $\alpha_1 u_1 \equiv u_2 \pmod{\alpha_2}$ has no solution $u_1,u_2 \in M$ as $u_1,u_2$ would have to be quadratic residues while $\alpha_1$ is not.

By Lemma \ref{lem:localization}, $E(G(x))$ does not have 1 in the stable range, and by Theorem \ref{thm:rank-one-char} $G(x)$ is not cancellable. Note when we apply this last theorem that $G(x)$ is not isomorphic to $\mathbb{Z}$ (see Remark \ref{rem:z}).
\end{proof}

The previous two claims complete the proof.
\end{proof}

\section{The Property of Having 1 in the Stable Range}\label{sec:sr1}

It is not hard to see, using examples (i) and (ii) from the introduction, that the property of having 1 in the stable range is $\Pi^0_2$ $m$-complete. One proves this by reducing $(\text{Fin},\text{Cof})$ to the index set of the computable rings with $1$ in the stable range, and then using the fact that $(\Sigma^0_2,\Pi^0_2) \leq_1 (\text{Fin},\text{Cof})$.

However, this is somewhat misleading. Indeed, given a subgroup $G$ of $\mathbb{Q}$, $E(G)$ is a $\Pi^0_2$ subring of $\mathbb{Q}$. Asking whether $E(G)$ inverts finitely many primes is $\Sigma^0_4$; however, asking whether $E(G)$ inverts cofinitely many primes is also $\Sigma^0_4$. Thus examples (i) and (ii) cannot be used to show that the cancellation property is $\Pi^0_4$ $m$-complete. 

The issue is that a computable ring isomorphic to a subring of $\mathbb{Q}$ is essentially a c.e.\ subring of $\mathbb{Q}$. A \textit{strongly computable subring of $\mathbb{Q}$} is a ring $R$ which is a computable subring of $\mathbb{Q}$. For a strongly computable ring $R$, the properties of inverting finitely primes and of inverting cofinitely many primes are both $\Sigma^0_2$. However, we get:

\begin{theorem}
	The set of strongly computable subrings of $\mathbb{Q}$ with 1 in the stable range is $\Pi^0_2$ $m$-complete within the strongly computable subrings of $\mathbb{Q}$.
\end{theorem}

Thus the complexity of checking whether 1 is in the stable range has everything to do with the condition in Definition \ref{def:sr1}, rather than with searching for inverses of elements.

Note that the complexity of the set of computable rings is $\Pi^0_2$, so we prove a completeness result within the strongly computable subrings of $\mathbb{Q}$, i.e., we ask that the $m$-reduction always produces a strongly computable subring of $\mathbb{Q}$.

\begin{proof}[Proof Sketch]
Given a $\Pi^0_2$ set $U$, write it as $\{x \mid (\forall i)(\exists j) S(x;i,j)\}$ where $S$ is computable.
We define $R(x) = \mathbb{Z}_{M_x}$ to be the strongly computable subring of $\mathbb{Q}$ which inverts the following computable set of primes:
\[ M_x = \{ p \mid p\in P_{i,j}, (\forall i'\le i) (\exists j'\le j) S(x;i',j') \} \]
Note that each $P_{i,j}$ is computable uniformly (in fact, the construction of the $P_{i,j}$ in the previous section obtains a strong index for $P_{i,j}$ as a finite set), and we can compute whether $P_{i,j} \subseteq M_x$. Thus $R(x)$ is strongly computable.

We can then use the arguments of Claims \ref{cl:2} and \ref{cl:3} to show that $x$ is in the $\Pi^0_2$ set $U$ if and only if $R(x)$ has 1 in the stable range.
\end{proof}

\section{Infinite case}

We now move on to groups which are not finitely generated. As mentioned in the introduction, we actually have no upper bound on the complexity of the index set. The best result we have towards an upper bound---but which still does not obtain an upper bound---is a simple Shoenfield absoluteness argument. We will use the following lemma:
\begin{lemma}[Essentially Barwise \cite{Barwise73}, see Lemma 2.9 of \cite{KnightSchweberMontalban16}]\label{lem:bf}
Let $G$ and $H$ be groups (or any other type of structure). The following are equivalent:
\begin{enumerate}
	\item $G$ and $H$ are back-and-forth equivalent,
	\item in every generic extension where $G$ and $H$ are countable, they are isomorphic,
	\item in some generic extension where $G$ and $H$ are countable, they are isomorphic.
\end{enumerate}
\end{lemma}

\noindent $G$ and $H$ are \emph{back-and-forth equivalent} if there is a relation $\sim$ on finite tuples from $\mc{A}$ and $\mc{B}$ such that $\varnothing \sim \varnothing$ and for every $\bar{a} \sim \bar{b}$:
\begin{itemize}
	\item $\bar{a}$ and $\bar{b}$ have the same atomic type,
	\item for all $a'$, there is $b'$ such that $\bar{a}a' \sim \bar{b}b'$,
	\item for all $b'$, there is $a'$ such that $\bar{a}a' \sim \bar{b}b'$,
\end{itemize}
It is well-known that two countable structures which are back-and-forth equivalent are actually isomorphic.

We prove:

\begin{proposition}
	Let $A$ be a countable torsion-free abelian group. Then the following are equivalent:
	\begin{enumerate}
		\item $A$ cancels with countable groups: Whenever $G$ and $H$ are countable, and $A \oplus G \cong A \oplus H$, $G \cong H$.
		\item Whenever $G$ and $H$ are any groups and $A \oplus G \cong A \oplus H$, $G$ and $H$ are back-and-forth equivalent.
	\end{enumerate}
\end{proposition}

\begin{proof}
	(2) implies (1) follows easily from the fact that two countable groups which are back-and-forth equivalent are isomorphic. (1) implies (2) uses a forcing argument.
	
	Suppose that $A \oplus G \cong A \oplus H$. Let $V[G]$ be a generic extension in which $G$ and $H$ are countable. Note that (1) is a $\Pi^1_2$ property with parameter $A$, and so by Shoenfield absoluteness, (1) is also satisfied by $A$ in $V[G]$. Thus $G$ and $H$ are isomorphic in $V[G]$. By Lemma \ref{lem:bf} $G$ and $H$ are back-and-forth equivalent in $V$.
\end{proof}

(1) of the proposition is $\Pi^1_2$. Two countable groups which are back-and-forth equivalent are isomorphic, but the same is not true for uncountable groups. Two uncountable groups which are back-and-forth equivalent have the same collection of countable subgroups.

\bigskip{}

We now move to the problem of finding a lower bound. Here we are inspired by combining two results. First, note that if $G$ is any non-trivial group, then $G_\omega = G \oplus G \oplus G \oplus \cdots$ is not cancellable. Indeed, $G_\omega \oplus G$ is isomorphic to $G_\omega \oplus C_1$ where $C_1$ is the trivial group, but $G$ is non-trivial.

Second, Riggs's \cite{Riggs15} produced for each tree $T$ a group $G_T$ such that $G_T$ decomposes as a non-trivial direct sum if and only if $T$ has an infinite path. We will modify this argument to produce a group $H_T$ so that if $T$ has no infinite path then $H_T$ is indecomposable and cancellable, and so that if $T$ has an infinite path, then $H_T$ decomposes as $A \oplus B \oplus B \oplus B \oplus \cdots$ so that $G_T$ is not cancellable.

To show that $H_T$ is cancellable, we will ensure that $H_T$ has 1 in the stable range. Even in the non-finitely generated case, this is still a sufficient condition.

\begin{theorem}[Theorem 2 of \cite{Evans73}]
	Let $G$ be a torsion-free abelian group such that $E(G)$ has $1$ in the stable range. Then $G$ is cancellable.
\end{theorem}

\noindent This is not a necessary condition, as for example $\mathbb{Z}$ has the cancellation property but $E(\mathbb{Z}) = \mathbb{Z}$ does not have $1$ in the stable range.

We are now ready to prove:

{
\renewcommand{\thetheorem}{\ref{thm:main-inf}}
\begin{theorem}
	The index set of the class of cancellable torsion-free countable groups is $\Pi^1_1$ $m$-hard.
\end{theorem}
\addtocounter{theorem}{-1}
}
\begin{proof}
By \cite[Example 8.3(c)]{Arnold82}, there is a computable infinite and co-infinite subset $Q$ of primes such that $\Z_Q := \Z[\frac{1}{q^\infty}\mid q\in Q]$ has 1 in the stable range. Thus, we may find a computable family of disjoint sets of primes $\{t\}, \{P_i\}_{i\in\omega}, Q, \{R_i\}_{i\in \omega}$, such that $|P_i| = |Q| = |R_i| = \infty$, and $\Z_Q := \Z[\frac{1}{q^\infty}\mid q\in Q]$ has 1 in the stable range.

\paragraph{Construction} Given a tree $T$, we must construct $H_T$ so that $H_T$ is cancellable if and only if $T$ has no infinite path. The main idea is that we will construct countably many copies of Riggs' \cite[Theorem 6.1]{Riggs15} construction using the primes $P_i$ to get the groups $G_T^s$ for $s \in \omega$. However, we need to make the following modifications: First, every ``link" is a (disjoint) infinite set of primes instead of a single prime (indeed, only the links need to be replaced by infinite set of primes, but we will do this for all divisibility for consistency.) Second, we introduce a new element $z$ such that $t^\infty \mid z$, and introduce $\frac{z+x_0^s}{r}$ for every $s$ and $r\in R_s$; this ``links'' each $G^s_T$ together via $z$. Lastly, we will also make $q^\infty \mid h$ for every $q\in Q$ and $h\in H_T$.

We will write $\frac{x}{S^k}$ to mean the set of elements $\{\frac{x}{r^k}\mid r\in S\}$, and 
$\frac{x}{S^\infty}$ to mean the set of elements $\{\frac{x}{r^k}\mid r\in S, k\in \N\}$.

To construct $H_T$, we proceed as follows. The group $H_T$ will be a c.e.\ subgroup of $\mathbb{Q}^{\omega}$. Identify the standard $\mathbb{Z}$-basis of $\mathbb{Q}^{\omega}$ with the elements $\{x_i^s, y^s_i\mid i,s\in\omega\}$, $\{x^s_\sigma \mid \sigma \in \omega^{<\omega}, s\in \omega\}$, and $z$. $H_T$ will be generated by these elements together with certain other elements. We also introduce:
\begin{enumerate}
\item For $s, i\ge 0$ and $\sigma\in \omega^{<\omega}$, 
\[ \frac{x^s_i}{Q^\infty}, \frac{y^s_i}{Q^\infty}, \frac{x^s_\sigma}{Q^\infty}, \frac{z}{Q^\infty}, \text{ and } \frac{z}{t^\infty}. \]
\item For $s, i\ge 0$ and $\sigma\in \omega^{<\omega}$, 
\[ \frac{x^s_i}{P_{\langle 0,i\rangle}^\infty}, \frac{y^s_i}{P_{\langle 1,i\rangle}^\infty}, \text{ and } \frac{x^s_\sigma}{P_{\langle 2,\sigma\rangle}^\infty}. \]
\item For $s\ge 0$ and $0\le i<j$,
\[ \frac{x^s_i+x^s_j}{P_{\langle 3,\langle i,j \rangle\rangle}} \text{ and } \frac{y^s_i+y^s_j}{P_{\langle 4,\langle i,j  \rangle \rangle}}. \]
\item For $s, i\ge 0$ and $\sigma, \rho \in \omega^{<\omega}$, 
\[ \frac{x^s_i+x^s_\sigma}{P_{\langle 5,\langle i,\sigma \rangle \rangle}} \text{ and } \frac{x^s_\sigma+x^s_\rho}{P_{\langle 6,\langle \sigma,\rho \rangle \rangle}}. \]
\item For $s, i \ge 0$, 
\[ \frac{x^s_i+y^s_i}{P_{\langle 8,i\rangle}}. \]
\item For $s\ge 0$,
\[ \frac{z+x^s_0}{R_s}. \]
\end{enumerate}

\noindent We now enumerate $T$ in a way so that every node is enumerated only after its parent is enumerated. Whenever $\sigma$ gets enumerated, with $n = |\sigma|$, we introduce the following elements for every $s\ge 0$: (Note that in each case, the introduction of one of the element implies the existence of the other.)

\begin{enumerate}\setcounter{enumi}{6}
\item For $i \le n$, 
\[ \frac{y^s_i+x^s_{\sigma|_i}}{P_{\langle 1,i\rangle}^n} \text{ and }\frac{x^s_{\sigma|_i}}{P_{\langle 1,i\rangle}^n}. \]
\item For $i < n$, 
\[ \frac{(y^s_i+x^s_{\sigma|_i})+(y^s_n+x^s_{\sigma})}{P_{\langle 4,\langle i,n \rangle \rangle}} \text{ and } \frac{x^s_{\sigma|_i}+x^s_{\sigma}}{P_{\langle 4,\langle i,n \rangle \rangle}}.\]
\item With $n = |\sigma|$ as in the above paragraph, \[ \frac{y^s_n+x^s_\sigma}{P_{\langle 8,n\rangle}} \text{ and }\frac{x^s_n-x^s_\sigma}{P_{\langle 8,n\rangle}}.\]
\end{enumerate}
The $H_T$ is the subgroup of $\mathbb{Q}^\omega$ generated by all of these elements.


\paragraph{Verification} We need to show that $H_T$ is cancellable if and only if $T$ has no infinite path. 

\bigskip{}

($\Rightarrow$) Suppose first that $T$ has an infinite path $\pi$. We will show that $H_T$ is not cancellable by showing that $H_T \cong A\oplus \bigoplus_{i\in\omega}B$ for some subgroup $A$ and $B$. Let $B^s_\pi$ be the pure subgroup generated by the elements $y^s_i+x^s_{\pi|_i}$, and $A$ be the pure subgroup generated by the elements $z$, $x^s_i$, and $x^s_{\sigma}$ (for all $s$). We will prove the following claim:

\begin{claim}
$H_T = A\oplus \bigoplus_{s\in\omega}B^s_\pi$.
\end{claim}

Given the claim, the fact that $H_T \cong A\oplus \bigoplus_{i\in\omega}B$ for some $B$ follows easily. The only generator introduced in the construction that is dependent on $s$ is (7), but it lives in the $A$ component. Thus, the $B^s_\pi$ are isomorphic to each other by changing the superscript. Hence, $H_T \cong H_T \oplus \bigoplus_{i \in \omega} B^0_\pi$, and thus $H_T$ is not cancellable.

\begin{proof}[Proof of claim]
It is not hard to see that these are disjoint, so it suffices to show that every generator of $H_T$ can be decomposed as a sum of elements from these groups. This is clear for the generators that do not involve $y^s_i$. For the generators involving $y^s_i$, we have:
\begin{itemize}
\item For $s,i,k \ge 0$ and $q\in Q$,
\[ \frac{y^s_i}{q^k} = -\frac{x^s_{\sigma|_i}}{q^k} + \frac{y^s_i+x^s_{\sigma|_i}}{q^k}, \]
where the elements on the right hand side are introduced in (1).

\item For $s,i,k \ge 0$ and $p\in P_{\langle 1,i\rangle}$,
\[ \frac{y^s_i}{p^k} = -\frac{x^s_{\sigma|_i}}{p^k} + \frac{y^s_i+x^s_{\sigma|_i}}{p^k}, \]
where the elements on the right hand side are introduced in (7) for $\sigma = \pi|_{\max(i,k)}$.

\item For $s \ge 0$, $0\le i < j$, and $p \in P_{\langle 4,\langle i,j  \rangle \rangle}$,
\[ \frac{y^s_i+y^s_j}{p} 
= \frac{(y^s_i+x^s_{\sigma_i})+(y^s_j+x^s_\sigma)}{p} - \frac{x^s_{\sigma_i}+x^s_\sigma}{p}
,\]
where the elements on the right hand side are introduced in (8) for $\sigma = \pi|_j$.


\item For $s, i \ge 0$ and $p \in P_{\langle 8,i\rangle}$,
\[ \frac{x^s_i+y^s_i}{p}
= \frac{y^s_i+x^s_\sigma}{p}+\frac{x^s_i-x^s_\sigma}{p}
,\]
where the elements on the right hand side are introduced in (9) for $\sigma = \pi|_i$.

\item For the generators introduced in (7) to (9), the difference (or sum for (9)) of the two elements have the form that we have already discussed, and the second element, if exists, is in $A$. Thus, the first element, being a difference (or sum for (9)) of two decomposable elements, is also decomposable. \qedhere
\end{itemize}
\end{proof}

\bigskip{}

($\Leftarrow$) We now suppose the $T$ has no infinite path. We will show that $H_T$ is cancellable by showing that $E(H_T)\cong \Z_Q$ which, by choice of $Q$, has 1 in the stable range.

Let $\phi: H_T \to H_T$ be an endomorphism. We first observe that as any prime in $Q$ divides every element $h\in H_T$, a copy of $\Z_Q$ embeds in $E(H_T)$; given $k \in \Z_Q$, the multiply-by-$k$ map is an endomorphism. Now consider the possible values of $\phi(z)$. Since the subgroup of elements infinitely divisible by $t$ is the pure subgroup generated by $z$, and $z$ is only divisible by $t$ and the primes in $Q$, we must have $\phi(z) = kz$ for some $k \in \Z_{Q\cup\{t\}}$. 

Now consider $\phi(x_0^s)$. The set of elements infinitely divisible by the primes in $P_{\langle 0,0 \rangle}$ is the pure subgroup generated by $\{x_0^s\}_{s\in\omega}$. Thus, we must have that $\phi(x_0^s)$ is a linear combination of $x_0^{s_0}, x_0^{s_1}, \dots, x_0^{s_\alpha}$. However, $\phi(x_0^s) - kx_0^s = \phi(x_0^s+z) - \phi(z) -kx_0^s = \phi(x_0^s+z) - k(z+x_0^s)$ must be divisible by all the primes in $R_s$. As the only elements divisible by all the primes in $R_s$ are those in the pure subgroup generated by $z+x^s_0$, this forces $\phi(x_0^s) = kx_0^s$ for every $s$.

The above argument can also be used to show that $\phi(x_n^s) = kx_n^s$ and $\phi(x_\sigma^s) = kx_\sigma^s$ for all $n,s\in \omega$ and $\sigma \in \omega^{<\omega}$. Indeed, as the set of elements infinitely divisible by the primes in $P_{\langle 0,n\rangle}$ is the pure subgroup generated by $\{x_n^s\}_{s\in\omega}$, we must have that $\phi(x^s_n)$ is a linear combination of $x_n^{s_0}, x_n^{s_1}, \dots, x_n^{s_\alpha}$. However, $\phi(x_n^s) - kx_n^s = \phi(x_0^s+x_n^s) - k(x_0^s+x_n^s)$ must be divisible by all the primes in $P_{\langle 3,\langle 0,n\rangle \rangle}$. As the only elements divisible by all the primes in $P_{\langle 3,\langle 0,n\rangle \rangle}$ are the subgroup generated by $x_0^{r}+x_n^{r}$ for $r \in \omega$, this forces $\phi(x_n^s) = kx_n^s$ for every $n,s$. Similarly, we must have $\phi(x_\sigma^s) = kx_\sigma^s$ for all $\sigma$ and $s$.

Now we consider $\phi(y_n^s)$. As $y_n^s$ is infinitely divisible by the primes in $P_{\langle 1,n \rangle}$, we must have that $\phi(y_n^s)$ is also divisible by the primes in $P_{\langle 1,n \rangle}$. In the construction, the only time we introduce $P_{\langle 1,n \rangle}$ divisibility is to $y_n^r$ and possibly $x_\sigma^r$ for some $|\sigma| = n$ (and their linear combinations.) Thus, $\phi(y_n^s)$ must be a linear combination of $y_n^r$ and $x_\sigma^r$ for some $r\in \omega$ and $|\sigma| = n$.

However, $\phi(y_n^s) - ky_n^s = \phi(y_n^s + x_n^s) -k(y_n^s + x_n^s)$ is divisible by all the primes in $P_{\langle 8,n\rangle}$. Thus, $\phi(y_n^s) - ky_n^s$ must be a linear combination of $y_n^r + x_n^r$ and $y_n^r+x_\sigma^r$ for some $r\in \omega$ and $|\sigma| = n$. 

Combining the previous two paragraphs, we have that $\phi(y_n^s) = ky_n^s + \sum \ell_{r,\sigma}(y^r_n+x^r_\sigma)$ for every $n,s\in \omega$.

\begin{claim}
If $\phi(y_n^s) \neq ky_n^s$ for some $n$ and $s$, then there is an infinite path through $T$.
\end{claim}

\begin{proof}
Assume that $\phi(y_n^s) \neq ky_n^s$ for some $n$ and $s$. Since $\phi(y_n^s) \neq ky_n^s$, fix $\ell_{r_0,\sigma_0} \neq 0$. This means that $y^{r_0}_n+x^{r_0}_{\sigma_0}$ is divisible by primes in $P_{\langle 8,n\rangle}$, so $\sigma_0 \in T$. 

Consider $\phi(y_{n+1}^s) = ky_{n+1}^s + \sum \ell'_{r,\rho}(y^{r}_{n+1}+x^{r}_{\rho})$. As $\phi(y_{n+1}^s+y_{n}^s)$ is divisible by the primes in $P_{\langle 4,\langle n,n+1\rangle\rangle}$, we have that $\phi(y_{n+1}^s+y_{n}^s)$ is a linear combination of $y^r_n+y^r_{n+1}$ and $x^r_{\rho|_n}+x^r_{\rho}$ for some $|\rho| = n+1$. This means that $\sum\limits_{|\rho|= n+1} \ell'_{r,\rho}(y^{r}_{n+1}+x^{r}_{\rho}) + \sum\limits_{|\sigma| = n} \ell_{r,\sigma}(y^r_n+x^r_\sigma)$ is a linear combination of $y^r_n+y^r_{n+1}$ and $x^r_{\rho|_n}+x^r_{\rho}$ for some $|\rho| = n+1$, and hence $\ell_{r,\sigma} = \sum\limits_{\rho:\rho|_n = \sigma}\ell'_{r,\rho}$.

Now we may choose a $\rho = \sigma_1$ such that $\ell'_{r_0,\sigma_1} \neq 0$. This means that $y^{r_0}_{n+1}+x^{r_0}_{\sigma_1}$ is divisible by primes in $P_{\langle 8,n\rangle}$, which implies that $\sigma_0 \subset\sigma_1 \in T$. Having $\ell'_{r_0,\sigma_1} \neq 0$ also implies that $\phi(y_{n+1}^s) \neq ky_{n+1}^s$. By applying the same argument iteratively, we may find $\sigma_0 \subset \sigma_1 \subset \sigma_2 \subset \dots$, which is an infinite path in $T$.
\end{proof}

By the claim, as there is no infinite path through $T$, we have $\phi(y_n^s) = ky_n^s$ for all $n,s\in\omega$. Thus $\phi(g) = kg$ for all $g\in H_T$. However, the only primes dividing all elements in $H_T$ are the primes in $Q$. Thus $k\in \Z_Q$, and $E(H_T) = \Z_Q$, which has 1 in the stable range. So $H_T$ is cancellable.
\end{proof}

\bibliography{References}

\begin{thebibliography}{KMS16}

\bibitem[Arn82]{Arnold82}
David~M. Arnold.
\newblock {\em Finite rank torsion free abelian groups and rings}, volume 931
  of {\em Lecture Notes in Mathematics}.
\newblock Springer-Verlag, Berlin-New York, 1982.

\bibitem[Bar73]{Barwise73}
Jon Barwise.
\newblock Back and forth through infinitary logic.
\newblock pages 5--34. MAA Studies in Math., Vol. 8, 1973.

\bibitem[Coh56]{Cohn56}
P.~M. Cohn.
\newblock The complement of a finitely generated direct summand of an abelian
  group.
\newblock {\em Proc. Amer. Math. Soc.}, 7:520--521, 1956.

\bibitem[Cra65]{Crawley65}
Peter Crawley.
\newblock The cancellation of torsion abelian groups in direct sums.
\newblock {\em J. Algebra}, 2:432--442, 1965.

\bibitem[EO67]{EstesOhm67}
Dennis Estes and Jack Ohm.
\newblock Stable range in commutative rings.
\newblock {\em J. Algebra}, 7:343--362, 1967.

\bibitem[Eva73]{Evans73}
E.~Graham Evans, Jr.
\newblock Krull-{S}chmidt and cancellation over local rings.
\newblock {\em Pacific J. Math.}, 46:115--121, 1973.

\bibitem[FL71]{FuchsLoonstra71}
L.~Fuchs and F.~Loonstra.
\newblock On the cancellation of modules in direct sums over {D}edekind
  domains.
\newblock {\em Nederl. Akad. Wetensch. Proc. Ser. A 74 = Indag. Math.},
  33:163--169, 1971.

\bibitem[Hs{\"u}62]{Hsu62}
Chin-shui Hs{\"u}.
\newblock Theorems on direct sums of modules.
\newblock {\em Proc. Amer. Math. Soc.}, 13:540--542, 1962.

\bibitem[Kap52]{Kaplansky52}
Irving Kaplansky.
\newblock Modules over {D}edekind rings and valuation rings.
\newblock {\em Trans. Amer. Math. Soc.}, 72:327--340, 1952.

\bibitem[Kap54]{Kaplansky54}
Irving Kaplansky.
\newblock {\em Infinite abelian groups}.
\newblock University of Michigan Press, Ann Arbor, 1954.

\bibitem[KMS16]{KnightSchweberMontalban16}
Julia Knight, Antonio Montalb\'an, and Noah Schweber.
\newblock Computable structures in generic extensions.
\newblock {\em J. Symb. Log.}, 81(3):814--832, 2016.

\bibitem[Rig15]{Riggs15}
Kyle Riggs.
\newblock The decomposability problem for torsion-free abelian groups is
  analytic-complete.
\newblock {\em Proc. Amer. Math. Soc.}, 143(8):3631--3640, 2015.

\bibitem[RY61]{RotmanYen61}
Joseph Rotman and Ti~Yen.
\newblock Modules over a complete discrete valuation ring.
\newblock {\em Trans. Amer. Math. Soc.}, 98:242--254, 1961.

\bibitem[Wal56]{Walker56}
Elbert~A. Walker.
\newblock Cancellation in direct sums of groups.
\newblock {\em Proc. Amer. Math. Soc.}, 7:898--902, 1956.

\end{thebibliography}
\bibliographystyle{alpha}

\end{document}